\documentclass{amsart}
\usepackage{amsmath,amsfonts,amssymb,enumerate}
\usepackage{hyperref}

\newcommand{\reference}{}
\newtheorem{thm}{Theorem}[section]
\newtheorem*{thmref}{Theorem~\reference}
\newtheorem*{propref}{Proposition~\reference}
\newtheorem*{corolref}{Corollary~\reference}
\newtheorem{lemma}[thm]{Lemma}

\newtheorem{prop}[thm]{Proposition}

\theoremstyle{definition}
\newtheorem{defi}[thm]{Definition}
\theoremstyle{remark}
\newtheorem{remark}[thm]{Remark}

\newcommand{\CP}{\mathbb{CP}}

\DeclareMathOperator{\tr}{tr}
\DeclareMathOperator{\End}{End}
\DeclareMathOperator{\id}{id}
\DeclareMathOperator{\Rm}{Rm}
\DeclareMathOperator{\Ric}{Ric}
\renewcommand{\div}{\dive}
\DeclareMathOperator{\dive}{div}

\DeclareMathOperator{\vol}{vol}
\renewcommand{\phi}{\varphi}
\renewcommand{\epsilon}{\varepsilon}

\renewcommand{\bar}{\overline}
\renewcommand{\tilde}{\widetilde}
\newcommand{\lieder}{\mathcal L}
\newcommand{\de}{\partial}

\newcommand{\Reeb}{\xi}
\newcommand{\I}{f}
\newcommand{\eps}{\epsilon}
\renewcommand{\o}{\omega}

\title[Second variation for Legendrians in pseudo-Sasakian manifolds]{Second variation for L-minimal Legendrian submanifolds in pseudo-Sasakian manifolds}
\author{David Petrecca \and Lars Sch\"afer}

\address{Institut f\"ur Differentialgeometrie, Leibniz Universit\"at Hannover, Welfengarten 1, Hanover, Germany.} \email{petrecca at math.uni-hannover.de, schaefer at math.uni-hannover.de}


\begin{document}

\maketitle

\begin{abstract}
In this paper we provide the second variation formula for L-minimal Lagrangian submanifolds in a pseudo-Sasakian manifold. 
We apply it to the case of Lorentzian-Sasakian manifolds and relate the L-stability of L-minimal Legendrian submanifolds in a Sasakian manifold $M$ to their L-stability in an associated Lorentzian-Sasakian structure on $M$.
\end{abstract}

\section*{Introduction}

Let $(M, g)$ be a Riemannian manifold. If $f:L \rightarrow M$ is a Riemannian submanifold, then it is called \emph{minimal} if $t=0$ is a critical point of the volume functional for all deformations $f_t: L \rightarrow M$ with $t \in (-\eps, \eps)$ and $f_0 = f$. Equivalently, $L$ is minimal if and only if its mean curvature vector vanishes. The submanifold is called \emph{stable} if $t=0$ is actually a minimum, that is if the second derivative of the volume functional at $t=0$ is nonnegative.

The explicit expressions of the first and second derivatives of the volume are standard and can be found, for instance, in \cite{simons}.

When $(M, \omega)$  is K\"ahler of real dimension $2n$, it is natural to study the above problem restricted to minimal \emph{Lagrangian} submanifolds, namely $n$-dimensional submanifolds that are minimal in the Riemannian geometric sense and where $\omega$ vanishes.

Let us restrict ourselves to deformations that keep $L$ Lagrangian, namely such that $f_t^* \omega = 0$. Infinitesimally this can be seen in the fact that $\lieder_X \omega = d(\iota_X \omega) = 0$, where $X$ is the normal component of the derivative of $f_t$. These deformations are called \emph{Lagrangian}.

In \cite{Oh_InvMath}, Oh has introduced the notion of \emph{Hamiltonian stability}. A minimal Lagrangian submanifold is Hamiltonian stable (H-stable) if its volume is a minimum among all infinitesimal \emph{Hamiltonian} deformations, namely given infinitesimally by normal fields $X$ such that $\iota_X \omega$ is \emph{exact}, i.e. Hamiltonian vector fields.

He then computes the Jacobi operator of a minimal Lagrangian submanifold and applies his second variation formula to provide a stability criterion for a submanifold $L$ in a K\"ahler-Einstein ambient in terms of the first eigenvalue $\lambda_1(L)$ of the Laplacian on $L$. Namely $L$ is H-stable if, and only if, $\lambda_1(L)$ is greater or equal than the Einstein constant of $M$.

There are several examples of minimal H-stable submanifolds of $\CP^n$ or other Hermitian symmetric spaces that are not stable in the usual sense. A survey of results and techniques, mostly for the homogeneous case, can be found in Ohnita's paper \cite{Ohnita_surv}.

A slight generalization of minimal Lagrangian submanifolds are \emph{H-minimal} ones, namely Lagrangians that extremize the volume under all \emph{Hamiltonian} variations or, equivalently, if the mean curvature vector is $L^2$-orthogonal to all Hamiltonian vector fields, see \cite{Oh_Hmin}.

The odd dimensional counterpart of K\"ahler geometry is Sasakian geometry, that merges together Riemannian, contact and CR structures. A natural contact geometric object analogous to Lagrangian submanifolds are Legendrian submanifolds.

A submanifold $f: L^n \rightarrow (M^{2n+1}, \eta)$ of a contact manifold is \emph{Legendrian} if $f^*\eta = 0$ and a deformation $f_t$ that preserves the Legendre condition is called \emph{Legendrian}. Infinitesimally, it translates into having a variation field that is a contactomorphism.

Oh's notion of H-stability is here replaced by Legendrian stability (L-stability), namely when the second derivative of the volume functional is nonnegative for all contact vector fields.

The computation of the second variation of minimal Legendrian submanifolds in Sasakian manifolds has been provided by Ono \cite{Ono}, along as a stability criterion -- for a $\eta$-Sasaki-Einstein ambient -- in terms of the spectrum of the Laplacian. Namely if the ambient Ricci tensor satisfies $\Ric = A g + (2n-A) \eta \otimes \eta$, then  $L$ is L-stable if, and only if, $\lambda_1(L) \geq A+2$, the K\"ahler-Einstein constant of the transverse metric.

Using Ono's expression of the second variation and the known properties of the Jacobi operator, Calamai and the first author \cite{calapet} were able to construct eigenfunctions of the Laplacian with eigenvalue $A+2$, under the assumption of the presence of nontrivial Sasaki ambient automorphisms.

Ono's work has been generalized by Kajigaya \cite{Kajigaya}, who has introduced the notion of L-minimal Legendrian submanifolds, namely the ones that are stationary points of the volume under Legendrian deformations and computed their second variation. In this case, a criterion involving the spectrum of the Laplacian cannot be provided in dimension (of the ambient manifold) greater than three.

The minimality condition extends of course to the pseudo-Riemannian setting and is treated in Anciaux's monograph \cite{anciaux_book}. The compatible combination of a pseudo-Riemannian metric and a complex structure leads to the notion of pseudo-K\"ahler structures that, being symplectic, allow us to speak about Lagrangian submanifolds. A, up to a certain point, similar structure of symplectic pseudo-Riemann\-ian manifold is given by para-K\"ahler ones, for which we refer to \cite{survey_AMT}, \cite{paracomplex_survey96}.

The study of the Hamiltonian stability of minimal Lagrangian in pseudo- and para-K\"ahler manifolds has been done by Anciaux and Georgiou \cite{anc_geo}, where they compute the second variation of such submanifolds and give a stability criterion analogous to Oh's in case these are space-like.

In this paper we treat the analogous problem for pseudo-Sasakian manifolds. These structures have been introduced by Takahashi in \cite{Takahashi} and consist in normal almost contact structures endowed with compatible pseudo-Riemannian metrics.

Our main result is the following (see Section \ref{sec:legendrian}).
\renewcommand{\reference}{\ref{thm:2var}}
\begin{thmref}
 Let $L$ be a L-minimal Legendrian submanifold, possibly with boundary $\partial L$, of a pseudo-Sasakian manifold $(M, \eta, \xi, g,\phi, \eps)$ with $\eps = |\xi|^2 = \pm 1$.
 
 Then, in the normal Legendrian direction $V = f \xi + \frac 1 2 \phi \nabla f$ vanishing on $\partial L$, the second variation of the volume is
\begin{align*}
 \frac{d^2}{dt^2}\biggr |_{t=0} \vol(L_t) = \frac 1 4 \int_L \biggl \{ 	&(\Delta f)^2 -2 \eps |\nabla f|^2 - \bar \Ric(\phi \nabla f, \phi \nabla f) \\
									&-2 g(H, h(\nabla f, \nabla f) + g(H, \phi \nabla f)^2  \biggr \} dv_0
\end{align*}
where $H$ is the mean curvature vector, $\bar \Ric$ is the Ricci tensor of $(M, g)$ and $dv_0$ is the volume form of $(L, g)$.
\end{thmref}

In the special case when $L$ is minimal ($H=0$) and $g$ is $\eta$-Einstein ($\bar \Ric = A g + (2n+\eps A) \eta \otimes \eta$), the formula above simplifies to
\begin{align*}
 \frac{d^2}{dt^2}\biggr |_{t=0} \vol(L_t)	&= \frac 1 4 \int_L \biggl \{  |\Delta f|^2 - (A+2 \eps)|\nabla f|^2 \biggr \} dv_0
\end{align*}
and we are able to give the following stability criterion in case $L$ is space-like.
\renewcommand{\reference}{\ref{prop:critlambda1}}
\begin{propref} 
 The minimal space-like Legendrian $L$ in the pseudo-Sasaki $\eta$-Einstein manifold $M$ is Legendrian stable if and only if its first eigenvalue of the Laplacian on functions $\lambda_1(L)$ satisfies
 \begin{equation} \label{est:lambda1}
  \lambda_1(L) \geq A+2 \eps.
 \end{equation}
\end{propref}

For $\eps = 1$ we reobtain Ono's formula and stability criterion.

Concerning usual stability in the pseudo-Riemannian case, it is known that every minimal submanifold is always unstable if the ambient metric is indefinite on its tangent or normal bundle, see \cite[Thm.~37]{anciaux_book}. In contrast, we have the following.
\renewcommand{\reference}{\!\!}
\begin{corolref}
 If $A+2\eps \leq 0$, then every minimal Legendrian submanifold is Legendrian stable. In particular this holds in the pseudo-Sasaki-Einstein case (with $A=-2n$).
\end{corolref}

Then we focus on Lorentzian-Sasakian manifolds, namely when the signature is $(2n,1)$ and $\eps = -1$. They appeared in \cite{Baum}, \cite{Bohle} in the study of twistor and Killing spinors on Lorentzian manifolds. Later their study has been proposed in Sasakian geometry, see \cite{BGM_eta} or \cite[Sect.~11.8.1]{monoBG}. In particular it is proved in \cite{BGM_eta} that every negative Sasakian manifold admits a Lorentzian-Sasaki-Einstein metric and conversely.

In Subsection \ref{sec:tanno} we consider these deformations that map every Sasakian structure to a Lorentzian-Sasakian one. They generalize the well-known $D$-homothetic deformations of Tanno \cite{Tanno}.

We then prove that for every minimal Lagrangian submanifold $L$ in a Sasakian manifold $M$ is Legendrian stable if, and only if, it is in the associated Lorentzian-Sasakian structure on $M$.

%there exists a Lorentzian-Sasakian structure on $M$ with respect to which $L$ is Legendrian stable. This applies in particular to the standard Sasakian sphere, in which it is known that every L-minimal Legendrian is unstable, see \cite{Kajigaya}, \cite{Ono}.
%
%Finally, in Subsection \ref{sec:Lmin} we consider the general case of L-minimal Legendrian submanifolds and prove an analogous stability criterion and study deformations for three-dimensional Sasakian space forms.

 \section{Pseudo-Sasakian manifolds}
In this section we recall the definition and main properties of pseudo-Sasakian structures, following \cite{Takahashi}.

Let $M^{2n+1}$ be a differentiable manifold and let $\xi$ be a vector field on $M$, $\eta$ a $1$-form and $\phi$ a section of $\End(TM)$.
 
 Then the triple $(\xi, \eta, \phi)$ is an \emph{almost contact structure} if $\eta(\xi) = 1$ and $\phi^2 = - \id + \eta \otimes \xi$, see e.g. \cite{Blair}. If $g$ is a pseudo-Riemannian metric, then we have the following.
 
 \begin{defi}
  The tuple $(\xi, \eta, \phi, g, \eps)$ is an \emph{almost contact metric structure} if $(\xi, \eta, \phi)$ is an almost contact structure and the following compatibility relations hold
  
  \begin{enumerate}
  \item $g(\xi, \xi) = \eps \in \{ \pm 1 \}$;
  \item $\eta(X) = \eps g(\xi, X)$;
  \item $g(\phi X, \phi Y) = g(X,Y) - \eps \eta(X) \eta(Y)$.
  \end{enumerate}
 \end{defi}
 A tuple as above is a \emph{contact metric structure} if\footnote{Unlike Takahashi, in this paper we use the convention $d\eta(X,Y) = X \eta(Y) - Y \eta(X) - \eta([X,Y])$.} $d\eta = 2g(\phi \cdot, \cdot)$.
 \begin{defi}
  A contact metric structure is \emph{normal} or \emph{Sasakian} if
  \begin{equation} \label{defnabla}
   (\nabla_X \phi) Y = \eps \eta(Y)X-g(X,Y)\xi
  \end{equation}
where $\nabla$ is the Levi-Civita connection of $g$.
 \end{defi}
 For an almost contact metric structure we have the following.
 \begin{prop} \label{prop:nablaxi}
  If the identity \eqref{defnabla} holds, then $\nabla \xi = \eps \phi$, the field $\xi$ is Killing and the structure is contact metric.
 \end{prop}
 In this paper we focus on a special kind of pseudo-Sasakian manifolds, namely Lorentzian Sasakian. They are characterized by their signature $(2n, 1)$ and $\eps = -1$, see e.g. \cite{Baum},\cite{Bohle}.
  
Before giving some properties of pseudo-Sasakian manifolds we need to fix
a sign convention for the curvature tensor $R$ of a connection $D$ on a vector
bundle $E  \rightarrow M$
\[
R(X,Y) \sigma =  D_X D_Y \sigma - D_Y D_X \sigma - D_{[X,Y]} \sigma \text{ for } X,Y \in \Gamma(TM) \text{ and } \sigma \in \Gamma(E).
\]

We have the following properties, some of them proved in \cite{schaefer} for the Lorentzian case.
\begin{lemma} 
Let $(M,g, \Reeb, \eta, \eps)$ be a pseudo-Sasakian manifold, then for $X,Y,Z \in TM$ one has 

\begin{align}
\phi^2 X				&= -X+\eta(X)\Reeb, \label{eq_Lem_Sas1} \\
(\bar \nabla_X \phi)Y 			&= -g(X,Y) \Reeb + \eps \eta(Y)X,  \label{eq_Lem_Sas1bis} \\
g(\phi X,Y)				&=- g(X, \phi Y),  \label{eq_Lem_Sas1bisbis} \\
\o(X,Y)					&= (\bar \nabla_X \eta) Y=g(\phi X, Y),  \label{eq_Lem_Sas2a} \\
(\bar \nabla_X \o)(Y,Z) 		&= \eps g(X,Z)\eta(Y) - \eps g(X,Y) \eta(Z), \label{eq_Lem_Sas3} \\ 
\bar\Rm(X,Y)\Reeb 			&= \eta(X)Y - \eta(Y)\, X, \label{eq_Lem_Sas6} \\ 
\bar\Rm(X,\xi,  \xi, Y) 		&= g(X,Y) -\eps \eta(X) \eta(Y),  \label{eq_Lem_Sas7}\\
\bar{\Ric}(\Reeb,\Reeb)			&= 2n, \label{Rici_in_Reeb} \\
\bar \Rm(X,Y) \phi Z			&= \phi \bar \Rm(X,Y)Z + \eps \biggl( -g(\phi Y, Z) X + g(\phi X, Z) Y \\ 
                                        &- g(Y,Z)\phi X + g(X,Z) \phi Y)\biggr) \label{eq_RmPhi}
\end{align}
where $\bar\Rm$ is the Riemann curvature tensor of $(M, g)$ and $\bar{\Ric}$ is its Ricci tensor.
\end{lemma}

\subsection{Legendrian submanifolds}

A submanifold $\I\,:\, L \rightarrow M$ of a contact manifold $(M^{2n+1}, \eta)$ is called \emph{
  horizontal} if it satisfies $\I^* \eta=0$. In particular, it follows  $\I^*
d\eta= \I^*\o=0.$ A  \emph{ Legendrian submanifold} is a maximally isotropic
submanifold $L^n$, i.e. a horizontal submanifold with $\dim
L=n$.  
	
Let us consider a smooth Riemannian immersion $ \I \,:\, L \rightarrow (M,g)$
into a Lorentzian manifold $(M,g),$  i.e. $ \I^*g$ defines a positive
definite metric on $L$. Then the \emph{second fundamental form} $h\in
\Gamma(T^*L \otimes T^*L\otimes NL),$ where $NL$ denotes the \emph{normal bundle} of $ \I \,:\, L
\rightarrow (M,g),$ is given by 
\[
h(X,Y)= \nabla_X(d\I)Y=(\I^*\bar \nabla)_X(d\I Y)- d\I(\nabla_XY),
\]
where $\bar \nabla$ and  $\nabla,$ resp. are the Levi-Civita connections of $g$ and $f^*g,$
resp.\footnote{To keep notation short, we later write $\bar \nabla $ for $f^*\bar \nabla$ and $g $ for $f^*g$.}
Further, for a section $\nu$ of $NL$ we define the normal connection $\nabla^\perp$ as the normal part and the shape operator $A_\nu$  as the tangential part of $\I^*\bar \nabla_X \nu,$ i.e. via 
\[
\I^*\bar \nabla_X \nu = \nabla^\perp_X \nu - A_\nu X \in NL \oplus \I_*(TL),
\]
where $A \in \Gamma(N^*L \otimes T^*L\otimes TL).$
The \emph{mean curvature} is defined as 
\[
H:= \tr_g^L h.
\]

Later we need Gauss' formula for pseudo-Riemannian submanifolds (see e.g. \cite[p.~100]{Oneill}).\footnote{Beware that O'Neill uses the opposite convention than ours for Riemannian curvature.}
\begin{lemma} \label{lemma:gauss}
The Riemann curvature tensor $\Rm$ of a pseudo-Riemannian submanifold $L$ in the pseudo-Riemannian manifold $(M,g)$ is related to the ambient curvature tensor $\bar \Rm$ and to the second fundamental form $h$ of the immersion by
\[
 \bar \Rm (A,B,C,D) = \Rm(A,B,C,D) - g( h(B,C), h(A,D)) + g( h(A,C), h(B,D))
\]
for vectors $A,B,C,D$ tangent to $L$.
\end{lemma}

Following \cite{Ono}, we give the following definition.

\begin{defi}
Let  $(M,g,\Reeb,\eta,\phi,\eps)$ be a pseudo-Sasakian manifold and $\I \colon L \rightarrow M$ a Legendrian immersion. 
A smooth family of immersions  $\{ f_t \}_{t\in (-\delta,\delta)}$  is called \emph{Legendrian deformation} of $L,$ 
if $f_t$ is Legendrian for all $t \in (-\delta,\delta)$ and it is $f_0=f.$
\end{defi}

By the curvature properties of pseudo-Sasakian metrics, we have the following.

\begin{lemma} \label{lemma:Rm}
For a Legendrian submanifold $L$  in a pseudo-Sasakian manifold $(M,g,\Reeb,\eta,\phi,\eps)$ and in a normal orthonormal frame $e_1, \ldots, e_n$ with $\eps_i= g(e_i,e_i)$, along the Legendrian submanifold $L$ one has
\begin{enumerate}
\item $\sum_{i=1}^n \eps_i \bar\Rm(\phi e_i, \xi,\xi, \phi e_i) = n,$
\item $\eps_i \bar\Rm(\phi e_i, \xi, \phi e_i, V_H) = 0, \mbox{ for } V_H \in \mathcal D.$
%\item $\bar\Rm(\xi, V, \xi, V) = -|V|^2 - \eta(V)^2$.
\end{enumerate}
\end{lemma}
\begin{proof}
The first is a consequence of \eqref{eq_Lem_Sas6}, i.e.
$\bar\Rm(\phi e_i, \xi,\xi, \phi e_i) = g(\phi e_i,\phi e_i).$
The second follows from \eqref{eq_Lem_Sas6} and $\eta(\phi e_i)=0.$
\end{proof}

We state a property of the second fundamental form of a Legendrian submanifold in a pseudo-Sasakian manifold, whose proof is basically the same as for its Riemannian counterpart; see \cite[Prop.~3.4]{Ono}.

\begin{lemma} \label{lemma:2ff}
 The second fundamental form $h$ of a Legendrian submanifold $L$ in a pseudo-Sasakian manifold $(M, g, \eta, \xi, \phi, \eps)$ satisfies the following properties.
 \begin{enumerate}
  \item $g(h(X,Y), \xi) = 0$ for all $X, Y$ tangent to $L$;
  \item $g(h(X,Y), \phi Z) = g(h(X,Z), \phi Y)$ for all $X,Y,Z$ tangent to $L$.
 \end{enumerate}
\end{lemma}

\section{L-minimal Legendrian submanifolds of pseudo-Sasakian manifolds} \label{sec:legendrian}
\subsection{The second variation along Legendrian deformations}

\begin{defi}
Let $ \I \colon L \rightarrow M$ be a Legendrian immersion into a pseudo-Sasakian manifold $(M,g,\Reeb,\eta,\phi,\eps).$ Then $\I$ is called \emph{Legendrian minimal}  or \emph{L-minimal} 
if
\[
\frac{d}{dt}\biggr |_{t=0} \vol(L_t)=0
\]
for all Legendrian deformations of $L_t$.
\end{defi}

Taking the normal field $V = \biggl(\frac{\de}{\de t} |_{t=0} f_t(\cdot)\biggr)^\perp$, then $f_t$ is Legendrian if, and only if, $\lieder_V \eta = 0$.% that is equivalent to $d \eta(X) = 2 \eps \alpha_X$, where $\alpha_V = - \frac 1 2 \eps \iota_V d\eta$.

From the known expression of the first variation (see e.g. \cite{anciaux_book}), namely 
\[
\frac{d}{dt}\biggr |_{t=0} \vol(L_t) = -\int_L g(X,H) dv_0,
\]
we see that L-minimality is equivalent to requiring $H$ to be $L^2$-orthogonal to all Legendrian vector fields.

For a normal field $V$ we write $V = f \xi + V_H$. It then is $\iota_{V_H} d\eta = 2(\phi V_H)^\flat$, implying, since $\lieder_V \eta = 0$, that $V_H = \frac 1 2 \phi \nabla f$.
%\noindent
For positive signature the following is due to \cite{Iriyeh}.
\begin{prop} \label{Lmin}
The immersion $ \I \colon L \rightarrow M$ of a manifold $L$  is L-minimal (with respect to variations fixing the boundary)
if and only if it is
\begin{equation}
\delta \alpha_H=0 \text{ or equivalently }\div(\phi H)=0 \label{eq_lmin} 
\end{equation} 
where $\alpha_H = d\eta(H, \cdot)$.
\end{prop}
\begin{proof}
In fact, the well-known formula for the first variation along the normal direction $X$ yields for variations with $d\alpha_X =d u$ for some function  $u$ on $L$
\begin{align*}
\frac{d}{dt}\biggr |_{t=0} \vol(L_t) &= - \int_L g(X, H) dv_0 =- \int_L g(\alpha_X, \alpha_H) dv_0 \\& =- \int_L g(d u, \alpha_H) dv_0 = - \int_L u \delta\alpha_H dv_0, 
\end{align*}
where we used that $X$ vanishes on the boundary. Since this vanishes for arbitrary  Legendrian variations we conclude \eqref{eq_lmin}.
\end{proof}
\begin{prop} \label{prop_second_Lmin}

 Let $(M, \eta, \xi, g, \phi, \eps)$ be a pseudo-Sasakian manifold and let $L$ be an L-minimal Legendrian submanifold, possibly with boundary. Then the second variation of the volume of $L$ under the normal direction $V = f \xi + \frac{1}{2}\phi \nabla f$, vanishing on $\partial L$, is
 \begin{align} \label{formula2var}
  \frac{d^2}{dt^2}\biggr |_{t=0} \vol(L_t) 	= \int_L \biggl \{ 	& \tr_g[(\nabla^\perp_{\cdot} V,\nabla^\perp_{\cdot} V) + \bar \Rm ({\cdot}, V, {\cdot}, V)]- g(A_V, A_V) \\
									&- \frac{1}{4} g(h(\nabla f, \nabla f), H) + g(H,V)^2 \biggr \} dv_0  \nonumber 
 \end{align}
where $L_t, \; t\in (-\delta,\delta), $ is a family of submanifolds with variational vector field $V$ and $dv_0$ is the volume form of the induced metric at $t=0$. 
%and $e_i, \, i =1, \ldots ,n$ is a local orthonormal  frame of $L.$
\end{prop}
\begin{proof}
For this proof in positive signature we refer to \cite{schoen_wolfson}. We fix  a local ($t$-dependent) frame $e_i, i =1, \ldots ,n,$ for $L$. 
One starts with the well-known formula for the first variation along the direction $X$
\[
\frac{d}{dt} \vol(L_t) = - \int_L g(X, H) dv_t
\]
where $H$ is the mean curvature vector of the family $L_t$ with variational vector field $X$ and derives this expression at $t=0$
\begin{equation*}
  \frac{d^2}{dt^2}\biggr |_{t=0} \vol(L_t)	 =  -\int_L \frac{d}{dt} g(H,X)|_{t=0} dv_0 + \int_L g(X, H)^2 dv_0
\end{equation*}
where we have used the well-known fact that $\frac{d}{dt}(dv_t) = - g(X, H) dv_t$.

Here we  write $g$ for the induced metric on $L_t,$ too. Further we also write $\bar \nabla$ for the pull-back of the Levi-Civita connection along the immersion $(-\eps,\eps) \times L \ni  (t,p)\mapsto f_t(p) \in M$.
%\noindent 
The first two terms are
\begin{equation*}
\frac{d}{dt} \biggr |_{t=0} g(H,X) = \frac{d g^{ij}}{dt}(0) g(h(e_i, e_j), X ) + g^{ij} g \biggl ( \frac{d}{dt}|_{t=0} h(e_i, e_j), X \biggr )
 \end{equation*}
where we recall 
 \begin{align*}
 \frac{d}{dt} g_{ij} & = X \cdot g(e_i, e_j) = g(\bar \nabla_X e_i,e_j) + g(e_i, \bar \nabla_X e_j) = g(\bar \nabla_{e_i} X, e_j) + g(e_i, \bar \nabla_{e_j} X) 
 \\&= -2 g(h(e_i, e_j), X) = -2g(A_X e_i, e_j),
 \end{align*}
 where we have used that $[e_k,X]=0$.

It is then $\frac{d}{dt} g^{ij}(0) = - g^{ia} g'_{ab} g^{bj}|_{t=0} = -\eps_i \eps_jg(h(e_i, e_j), X)$.

We have
\begin{align*}
 -\frac{d}{dt} \biggr |_{t=0} g(H,X)	 &= - 2\eps_i \eps_j ( g(h(e_i, e_j), X)g(\bar \nabla_{e_i} X, e_j) + \delta_{ij} \eps_i X \cdot g(\bar \nabla_{e_i} X, e_j) \\
					 &= -2\eps_i \eps_j g(h(e_i, e_j), X)^2 + \eps_i \biggl [ g(\bar \nabla_X \bar \nabla_{e_i} X, e_i) + g(\bar \nabla_{e_i} X, \bar \nabla_X e_i)  \biggr ]\\
					 &=  -2\eps_i \eps_j g(h(e_i, e_j), X)^2  + \eps_i \biggl [ \bar \Rm(X,e_i,X,e_i) + g( \bar \nabla_{e_i} \bar \nabla_X X, e_i) + |\bar \nabla_{e_i}X|^2 \biggr ].  
\end{align*}

On the other hand $A_X e_i = g(A_X e_i, e_j) \eps_j e_j$, so $\eps_i g(A_X e_i, A_X e_i) = g(h(e_i, e_j), X)^2 \eps_j$. So we have

\begin{align*}
 \frac{d^2}{dt^2}\biggr |_{t=0} \vol(L_t)= \int_L \biggl \{ & \eps_i |\nabla^\perp_{e_i} X|^2 + \eps_i \bar \Rm(X, e_i, X, e_i) - \eps_i |A_X e_i|^2 \\ 
                                                            &+ \div(\bar \nabla_X X)^T - g(\bar \nabla_X X, H) + g(X,H)^2 \biggr \} dv_0\
\end{align*}

The divergence term is $\int_L \div( \bar \nabla_X X)^T dv_0 = \int_{\partial L} g(\bar \nabla_X X, \nu) = 0$ since $X$ vanishes on $\partial L$.

Finally we compute
\begin{align*}
 g(\bar \nabla_X X, H) - \frac 1 4 g(h(\nabla f, \nabla f), H)	&= g(\phi \bar \nabla_X X, \phi H) - \frac 1 4 g(\phi h(\nabla f, \phi H),\nabla f) \\
								&= g(\bar \nabla_X \phi X, \phi H) - \frac 1 4 g(\phi h(\nabla f, \phi H),\nabla f) \\
								&= X \cdot g(\phi X,\phi H)- g(\phi X, \bar \nabla_X \phi H)- \frac 1 4 g(\phi h(\nabla f, \phi H),\nabla f) \\
								&= (\bar \nabla_X (\phi X)^\flat)(\phi H)- \frac 1 4 g(\phi \bar \nabla_{\phi H} \nabla f, \nabla f)\\
								&= (\bar \nabla_X (\phi X)^\flat)(\phi H)+ \frac 1 2 g(\bar \nabla_{\phi H}  X, \nabla f)\\
								&= (\bar \nabla_X (\phi X)^\flat)(\phi H)+ \frac 1 2 df (\bar \nabla_{\phi H}  X) \\
								&= \frac 1 2 (  X df(\phi H) - df (\bar \nabla_X \phi H) + df (\bar \nabla_{\phi H}  X))\\
								&= \phi H \cdot X \cdot f \\
								&= 0
\end{align*}
since $X$ is \emph{normal} to $L$. So we can conclude  \eqref{formula2var}.
\end{proof}

Let us now compute the first term of \eqref{formula2var}. We have

\begin{align*}
 \nabla_{e_i}^\perp V &= \biggl ( \bar \nabla_{e_i} (f \xi + \frac 1 2 \phi \nabla f) \biggr )^\perp\\
		      &= f (\bar  \nabla_{e_i} \xi)^\perp + e_i \cdot f \xi + \frac 1 2( \bar \nabla_{e_i}  \phi \nabla f)^\perp\\
		      &= \eps f \phi e_i + e_i \cdot f \xi + \frac 1 2 ( \phi \nabla_{e_i} \nabla f - g(e_i, \nabla f) \xi) \\
		      &= \eps f \phi e_i + \frac 1 2 e_i \cdot f \xi + \frac 1 2 \phi \nabla_{e_i} \nabla f
\end{align*}
so we get summing over $i$
\begin{align}
 \eps_i g(\nabla_{e_i}^\perp V,\nabla_{e_i}^\perp V)	&=   f^2 \eps_i |\phi e_i|^2 + \frac 1 4 (e_i f)^2 \eps_i \eps + \frac 1 4 |\nabla^2 f|^2 + \eps f g(e_i, \nabla_{e_i} \nabla f) \eps_i\\ \nonumber
							&=  n f^2 + \frac 1 4 \eps |\nabla f|^2 + \frac 1 4 |\nabla^2 f|^2 - \eps f \Delta f \label{1stterm}
\end{align}
where $|\nabla^2 f|^2 = \eps_i |\nabla_{e_i} \nabla f|^2$ is the norm of the Hessian of $f$.

Applying \eqref{eq_RmPhi} we obtain the following, that says we have the symmetries analogous to the ones of the K\"ahler curvature tensor.

\begin{lemma}
 We have $\sum_i \eps_i \bar \Rm(\phi e_i, V_H, \phi e_i, V_H) = \sum_i \eps_i \bar \Rm(e_i, \phi V_H, e_i, \phi V_H)$.
\end{lemma}

\begin{proof}
 Applying the identity \eqref{eq_RmPhi} we have
 \begin{align*}
 \eps_i \bar \Rm(\phi e_i, V_H, \phi e_i, V_H)	&= \eps_i g(\phi \bar \Rm (\phi e_i, V_H) e_i, V_H) + \eps_i \eps (-g(\phi V_H, e_i) g(\phi e_i, V_H) - g(e_i, e_i) g(V_H, V_H)) \\
						&= - \eps_i \bar \Rm(\phi e_i, V_H, e_i, \phi V_H) + \eps( |V_H|^2 - n |V_H|^2)\\
						&= - \eps_i \bar \Rm(e_i, \phi V_H, \phi e_i, V_H) + \eps(1-n)|V_H|^2\\
						&= - \eps_i g(\phi \bar \Rm(e_i, \phi V_H) e_i, V_H) \\
						&- \eps_i \eps (-g(\phi V_H, e_i) g(\phi e_i, V_H) - g(e_i, e_i) g(V_H, V_H))) + \eps(1-n)|V_H|^2\\
						&= \eps_i \bar \Rm(e_i, \phi V_H, e_i, \phi V_H) - \eps( |V_H|^2 - \eps_i g(e_i, e_i) |V_H|^2) + \eps(1-n)|V_H|^2\\
						&= \eps_i \bar \Rm(e_i, \phi V_H, e_i, \phi V_H). 
 \end{align*}
\end{proof}

So the second term in our second variation is

\begin{align*}
\eps_i \bar \Rm (e_i, V, e_i, V)	&=  -\bar \Ric(V,V) - \eps_i \bar \Rm(\phi e_i, V, \phi e_i, V) - \eps \bar \Rm(\xi, V, \xi, V) \\
					&= - \bar \Ric(V,V) - \eps_i f^2 \bar \Rm( \phi e_i, \xi, \phi e_i, \xi) - 2 \eps_i f \bar \Rm(\phi e_i, V_H, \phi e_i, \xi)\\
					&- \eps_i \bar \Rm(\phi e_i, V_H, \phi e_i, V_H) - \eps \bar \Rm(\xi, V, \xi, V)\\
					&= -\bar \Ric(V,V) + n f^2 - \eps_i \bar \Rm(\phi e_i, V_H, \phi e_i, V_H) + \eps(|V|^2 - \eps f^2) \\
					&= -\bar \Ric(V,V) + n f^2 - \eps_i \bar \Rm(\phi e_i, V_H, \phi e_i, V_H) + \eps|V_H|^2\\
					&= -\bar \Ric(V,V) + n f^2 -\eps_i \bar \Rm(e_i, \phi V_H, e_i, \phi V_H) + \eps |V_H|^2\\
					&= -\bar \Ric(V_H,V_H) - 2nf^2 + n f^2 + \eps |V_H|^2 \\
					&- \eps_i \biggl ( \Rm(e_i, \phi V_H, e_i, \phi V_H) - g(h(\phi V_h, e_i), h(\phi V_H, e_i)) + g(h(e_i, e_i), h(\phi V_H, \phi V_H)) \biggr ) \\
					&= -\bar \Ric(V_H,V_H) - nf^2 + \eps |V_H|^2 + \Ric(\phi V_H, \phi V_H) + g(A_V, A_V) - g(H, h(\phi V_H, \phi V_H))
\end{align*}
where in the last equality we have used Gauss' formula in Lemma \ref{lemma:gauss} and the fact that, from the definition of the shape operator, we have $A_V e_i =- (\nabla_{e_i} V)^T$ and $A_V = A_{V_H}$. We then compute that $A_{V_H} e_i = \phi h (e_i, \phi V_H)$ and hence

\[
 g(A_V, A_V) = \eps_i g( h(e_i, \phi V_H), h(e_i, \phi V_H)).
\]

So \eqref{formula2var} becomes
\begin{align*}
 \frac{d^2}{dt^2}\biggr |_{t=0} \vol(L_t)	& = \int_L \biggl \{ nf^2 + \frac 1 4 \eps |\nabla f|^2 + \frac 1 4 |\nabla^2 f|^2 - \eps f \Delta f\\
 								&-n f^2 - \frac 1 4 \bar \Ric(\phi \nabla f, \phi \nabla f) + \frac 1 4 \eps |\nabla f|^2 + \frac 1 4 \Ric (\nabla f, \nabla f)\\
 								&\underbrace{- g( H, h(\phi V_H, \phi V_H)}_{=- \frac 1 4 g(h(\nabla f, \nabla f), H)} \\
								&- \frac 1 4 g(h(\nabla f, \nabla f), H) + \frac 1 4 g(H, \phi \nabla f)^2 \biggr \} dv_0 \\
						&= \frac 1 4 \int_L \biggl \{ -2 \eps |\nabla f|^2 - \bar \Ric (\phi \nabla f, \phi \nabla f) + |\nabla^2 f|^2 + \Ric(\nabla f, \nabla f)\\
						&-2 g(H, h(\nabla f, \nabla f) + g(H, \phi \nabla f)^2 \biggr \} dv_0 
\end{align*}

We can group the Hessian term and the Ricci term by means of the pseudo-Riemannian Bochner formula\footnote{The formula on p.~609 of \cite{anc_geo} differs by a sign as they define $\Delta = \div (\nabla \cdot)$.} in the Appendix of \cite{anc_geo}: 
\[ \frac 1 2 \Delta |\nabla f|^2 = \Ric(\nabla f, \nabla f) - g(\nabla f, \nabla(\Delta f)) + |\nabla^2 f|^2 \] that, after integration, gives

\[
\int_L (\Delta f)^2 = \int_L ( \Ric(\nabla f, \nabla f) + |\nabla^2 f|^2)
\]
since $X= f \xi + \frac 1 2 \phi \nabla f$ vanishes on $\partial L$ and get the following.

\begin{thm} \label{thm:2var}
 Let $L$ be a L-minimal Legendrian submanifold, possibly with boundary $\partial L$, of a pseudo-Sasakian manifold $(M, \eta, \xi, g,\phi, \eps)$.
 
 Then, in the normal Legendrian direction $V = f \xi + \frac 1 2 \phi \nabla f$ vanishing on $\partial L$, the second variation of the volume is
\begin{align*}
 \frac{d^2}{dt^2}\biggr |_{t=0} \vol(L_t) = \frac 1 4 \int_L \biggl \{ 	&(\Delta f)^2 -2 \eps |\nabla f|^2 - \bar \Ric(\phi \nabla f, \phi \nabla f) \\
									&-2 g(H, h(\nabla f, \nabla f) + g(H, \phi \nabla f)^2  \biggr \} dv_0
\end{align*}
where $H$ is the mean curvature vector and $dv_0$ is the volume form of $(L, g)$.
\end{thm}

\subsection{The minimal case}
Let us now consider the more special case where $L$ is minimal and \emph{Riemannian} (i.e. $H=0$) and $M$ is $\eta$-Sasaki-Einstein, i.e. for some $A,B \in \mathbb R,$ it holds
\[
\bar \Ric = A g + B \eta \otimes \eta
\]
where it must be $B = 2n - \eps A$.

In this case our second variation formula reads

\begin{align*}
 \frac{d^2}{dt^2}\biggr |_{t=0} \vol(L_t)	&= \frac 1 4 \int_L \biggl \{  |\Delta f|^2 - (A+2 \eps)|\nabla f|^2 \biggr \} dv_0
\end{align*}
and we recall that $|df|^2 \geq 0$ being $L$ Riemannian.

Note that for the Riemannian case $\eps = 1$ we have reobtained the formula of \cite{Ono} (see also \cite{ohnita}).

In this case we have a sufficient condition for Legendrian stability coming from the positivity of the second term in the last expression.
\begin{prop}
 A minimal Legendrian $n$-submanifold in a pseudo-Sasakian $\eta$-Einstein manifold with constant $A$ is Legendrian stable if 
 \begin{equation} \label{eq:alwaysstable}
 A + 2 \eps \leq 0.
 \end{equation}
 In particular, if $M$ is Lorentzian-Sasaki-Einstein we have $A = -2n$ and $\eps = -1$ so $L$ is always Legendrian stable.
\end{prop}

With the same argument used in the Sasakian and K\"ahler case, using that $L$ is a Riemannian manifold so the space $C^\infty(L)$ admits a $L^2$-orthogonal decomposition given by the eigenspaces of the Laplacian, we can prove the following.

\begin{prop} \label{prop:critlambda1}
 The minimal space-like Legendrian $L$ in the pseudo-Sasaki $\eta$-Einstein manifold $M$ is Legendrian stable if and only if its first eigenvalue of the Laplacian on functions $\lambda_1(L)$ satisfies
 \begin{equation} \label{est:lambda1}
  \lambda_1(L) \geq A+2 \eps.
 \end{equation}

\end{prop}

\section{Lorentzian Sasakian manifolds}
 
\subsection{Tanno deformations} \label{sec:tanno}
The following is a generalization of the well-known Tanno deformations \cite{Tanno}. Starting with a Sasakian manifold  $(M, g, \eta, \xi, \phi)$ one defines for  fixed  $\alpha \in \mathbb R_+$ and $\beta := \alpha+\alpha^2$ 

\begin{equation} \label{gtilde}
\tilde g:=    \tilde g_\alpha := \alpha g - \beta \eta \otimes \eta.
\end{equation}
This is a Lorentzian metric, since it is
$$\tilde g (\xi,\xi)= \alpha g (\xi,\xi) -(\alpha^2+\alpha)=-\alpha^2.$$

For the proof of this Proposition we need the next Lemma.
\begin{lemma} \label{lemma:nabla}
If $\tilde \nabla$ is the Levi-Civita connection of $\tilde g_\alpha$ and $\nabla$ is the one of $g$, then we have
\begin{equation} \label{eq:diffcon}
 \tilde \nabla_X Y = \nabla_X Y - \alpha^{-1} \beta( \eta(X) \phi Y + \eta(Y) \phi X).
\end{equation}
\end{lemma}
%\noindent
For a proof in the case $\alpha=1$ and $\beta=2$ using Koszul's formula we refer to Proposition 3.3 of Brunetti and Pastore \cite{brunpast}. The Riemannian case is due to Tanno \cite{Tanno}, see also \cite[Chap.~7]{monoBG}. We remark a sign difference with \cite{brunpast} due to the opposite convention in the definition of the fundamental $2$-form.

\begin{proof} 
Define $\tilde \nabla$ by
\[
 \tilde \nabla_X Y = \nabla_X Y + S_X Y,
 \]
 where 
 \[
S_XY:= -\alpha^{-1} \beta \biggl[ \eta(X) \phi Y + \eta(Y) \phi X \biggr].
\]

The tensor field $S$ is symmetric, hence $ \tilde \nabla$ is a torsion-free connection.

We compute
\begin{eqnarray*}
(\tilde \nabla_X \tilde g) (Y,Z) 	&=&   \nabla_X \tilde g (Y,Z) - \tilde g (S_X Y,Z ) - \tilde g (Y,S_XZ) \\
 									&=& - \beta \biggl ( (\nabla_X\eta)(Y)\eta(Z) + (\nabla_X\eta)(Z)\eta(Y) \biggr )\\ 
									&  &+\alpha^{-1} \beta \, \tilde g\biggl(\biggl[ \eta(X) \phi Y + \eta(Y) \phi X \biggr], Z \biggr)\\
									&  &+\alpha^{-1} \beta \,\tilde g\biggl(Y, \biggl[ \eta(X) \phi Z + \eta(Z) \phi X \biggr] \biggr)\\
									&=& 0,
\end{eqnarray*}
where we have used that $\nabla_X \eta (Y) = g(\phi X, Y)$, as a consequence of Proposition \ref{prop:nablaxi}.

Hence, $ \tilde \nabla$ is metric for $\tilde g$ and has no torsion, so it coincides with the Levi-Civita connection of $\tilde g$.
\end{proof}

The behavior \eqref{eq:diffcon} of the Levi-Civita connection of $\tilde g_\alpha$ allows us to prove the following.
\begin{prop}\label{prop:lor}
 Let $(\eta, \xi, \phi, g)$ be a Sasakian structure. Then for $\alpha >0$ the new structure $(\alpha \eta, \alpha^{-1} \xi, \phi, \tilde g_\alpha)$ is Lorentzian Sasakian, where $\tilde g_\alpha$ is defined in \eqref{gtilde}.
\end{prop}

\begin{proof}
For completeness sake we prove this Proposition. Let $\tilde \xi = \alpha^{-1} \xi$.
First we observe, that $Z\in \{\xi,\tilde \xi\}$ satisfies $ \mathcal L_Z g =0 \mbox{ and  } \mathcal L_Z \eta =0$  
and as a result $ \mathcal L_Z \tilde g =0.$ Hence $\tilde \xi$ is a Killing vector field of length $-1.$  \\
Moreover,  for the second term of $\phi^2$    using $\tilde g (\Reeb,\Reeb)=-\alpha^2$ one has
$g(X,\Reeb)\Reeb = -\tilde g(X,\tilde \Reeb)\tilde \Reeb,$ which shows, that the relation \eqref{eq_Lem_Sas1} is satisfied. \\
Let us note, that it is 
$ \tilde \nabla_X Y = \nabla_X Y, \mbox{ for } X,Y \in \mathcal D$  and $\tilde \nabla_\Reeb \Reeb = \nabla_\Reeb \Reeb=0.$
In order to check  \eqref{defnabla} we observe
$$-g(X,Y) \Reeb   +g(Y,\Reeb)X = -g(X,Y) \Reeb =-\tilde g(X,Y)\tilde \Reeb, \mbox{ for } X,Y \in \mathcal D $$
and  $(\nabla_\Reeb \phi)\Reeb =0 = (\tilde\nabla_{\tilde\Reeb} \phi) \tilde \Reeb.$
It remains to compute the expression for $X \in \mathcal D$
$$(\nabla_X \phi)\Reeb = -\phi (\nabla_X \Reeb)= -\phi \biggl(\tilde \nabla_X \Reeb - \alpha^{-1} \beta( \phi X) \biggr) =  (\tilde \nabla_X \phi)\Reeb -\alpha^{-1} \beta X $$ 
and $ -g(X,\Reeb) \Reeb   +g(\Reeb,\Reeb)X = -\tilde g(\tilde \Reeb,\tilde \Reeb)X.$ This shows
\begin{eqnarray*}
(\tilde \nabla_X \phi) \tilde \Reeb &=& \alpha^{-2} \beta X -\alpha^{-1}\tilde g(\tilde \Reeb,\tilde \Reeb)X = 
\frac{\alpha +\alpha^2-\alpha}{\alpha^{2}}\tilde g(\tilde \Reeb,\tilde \Reeb)X\\&=&-\tilde g(X,\tilde \Reeb) \tilde\Reeb   +\tilde g(\tilde\Reeb,\tilde\Reeb)X
\end{eqnarray*}
and finishes the proof of Proposition \ref{prop:lor}, since the converse statement goes along the same lines.
\end{proof}
%\noindent

We can compute how the Ricci tensor behaves under these deformations.

\begin{prop} \label{prop:loreinst}
  Let $(M, g, \eta, \xi, \phi)$ be a Sasakian $\eta$-Einstein manifold with $\Ric = Ag + (2n+A) \eta \otimes \eta$ and let $\tilde g_\alpha$ as above. Then $\tilde g_\alpha$ is Lorentzian Sasakian $\eta$-Einstein with $\tilde \Ric = A_\alpha \tilde g_\alpha + (2n +A_\alpha) \eta \otimes \eta$  for $A_\alpha = \frac{A+2}{\alpha} + 2$.
 \end{prop}

From Lemma \ref{lemma:nabla} we obtain the following about the curvature tensor.
 
 \begin{lemma} \label{lemma:Rm}
 The curvature tensors $\tilde \Rm$ of $\tilde g_\alpha$ and $\Rm$ of $g$ are related by
 \[
  \tilde \Rm(X,Y)Z = \Rm(X,Y)Z + \alpha^{-1} \beta  \,  \biggl ( g(\phi Y, Z) \phi X - g(\phi X, Z) \phi Y - 2g(\phi X, Y) \phi Z \biggr)
 \]
 for $X,Y,Z$ in $\mathcal D$.
   \end{lemma}
\begin{proof}
 We compute for $X,Y,Z$ in $\mathcal  D$
 \begin{align*}
  \tilde \Rm(X,Y)Z	&= \tilde \nabla_X \nabla_Y Z - \tilde \nabla_Y \nabla_X Z - \tilde \nabla_{[X,Y]} Z \\
			&= \nabla_X \nabla_Y Z - \alpha^{-1} \beta \, \eta(\nabla_Y Z) \phi X - \nabla_Y \nabla_X Z \\ &~~~+ \alpha^{-1} \beta \, \eta(\nabla_X Z) \phi Y - \nabla_{[X,Y]} Z + \alpha^{-1} \beta \, \eta([X,Y]) \phi Z\\
			&= \Rm(X,Y)Z - \alpha^{-1} \beta  \, \eta(\nabla_Y Z) \phi X \\ &~~~+ \alpha^{-1} \beta\,  \eta(\nabla_X Z) \phi Y + \alpha^{-1} \beta \, \eta([X,Y]) \phi Z.
 \end{align*}
 %\noindent 
 We have that 
\[ 
 \eta(\nabla_Y Z) = Y \eta(Z) - (\nabla_Y \eta)(Z) = -g(\phi Y, Z)
\] 
  and similarly $\eta(\nabla_X Z) = -g(\phi X, Z)$. Moreover, one has 
\[  
  \eta([X,Y]) = -d\eta(X,Y) = -2g(\phi X, Y).
\]  
   So we have
 \begin{align*}
  \tilde \Rm(X,Y)Z &= \Rm(X,Y)Z + \alpha^{-1} \beta  \,  g(\phi Y, Z)\phi X \\&~~- \alpha^{-1} \beta  \,  g(\phi X, Z)\phi Y -2\,\alpha^{-1} \beta  \,  g(\phi X, Y)\phi Z.  
 \end{align*}
\end{proof}

\begin{proof}[Proof of Proposition \ref{prop:loreinst}]
Let $E_i$ be an orthonormal frame with respect to $g$ of $\mathcal D$ and let $\tilde E_i = \frac 1 {\sqrt \alpha} E_i$. We want to compute, for $X, Y$ in $ \mathcal D$
\begin{align*}
\tilde \Ric (X,Y) &= \tilde \Rm (X, \tilde E_i, \tilde E_i, Y) - \tilde \Rm (X, \tilde \xi, \tilde \xi, Y) \\
			&= \tilde g ( \tilde \Rm (X, \tilde E_i) \tilde E_i, Y) - \tilde g(X, Y) \\
			&= g ( \tilde \Rm (X,  E_i)  E_i, Y) - \tilde g(X, Y) \\
			&= \Rm (X,  E_i,  E_i, Y) + \frac{\beta}{\alpha}  \biggl ( -g (\phi X, E_i) g(\phi E_i, Y) - 2 g(\phi X, E_i) g( \phi E_i, Y) \biggr ) - \tilde g(X, Y) \\
			&=  \Rm (X,  E_i,  E_i, Y) + 3  \frac{\beta}{\alpha} g(\phi X, \phi Y)  - \tilde g(X, Y) \\
			&= \Ric (X, Y) - \Rm (X, \xi, \xi, Y) + 3  \frac{\beta}{\alpha} g(\phi X, \phi Y)  - \tilde g(X, Y) \\
			&= A g (X,Y) - g(X,Y) + 3  \frac{\beta}{\alpha} g(\phi X, \phi Y)  - \tilde g(X, Y) \\
			&= \biggl (\frac A \alpha - \frac 1 \alpha + 3  \frac{\beta}{\alpha^2} -1 \biggr) \tilde g(X,Y) \\
			&= \biggl (\frac{A+2}{\alpha} + 2 \biggr ) \tilde g(X,Y).
\end{align*}
Thus we have $A_\alpha = \frac{A+2}{\alpha} + 2$.
 \end{proof}
 %\noindent 
\subsection{Tanno deformations and Legendrian instability}
 Through the transformation of Proposition \ref{prop:lor} it is easy to see the following.
 
 \begin{prop} \label{prop:minleg}
 Let $L \subset M$ be an $n$-dimensional submanifold. Then it is minimal Legendrian with respect to $(M,g, \eta, \xi, \phi)$ 
if and only if it is also so with respect to $(M,\tilde g_\alpha, \tilde \eta, \tilde \xi, \phi)$.
 \end{prop}
 \begin{proof}
  The contact structure does not change. As for minimality, from Lemma \ref{lemma:nabla} we can write down the difference of the mean curvature vectors which turns out to be zero, 
since the restrictions of $\tilde \nabla$ and $\nabla$ to $L$ coincide.
 \end{proof}
%\noindent
 We emphasize the following  observation.
\begin{remark} \label{rmk:isom}
 The induced metrics $g|_L$ and $\tilde g_\alpha|_L = \alpha g|_L$ on $L$ are homothetic. In particular, their Hodge-de-Rham Laplacians $\Delta_L$ 
are related by $\tilde \Delta_L =\alpha^{-1}  \Delta_L $ and their first eigenvalues via $ \tilde \lambda_1(L) =\alpha^{-1}\lambda_1(L).$
\end{remark}

From Lemma \ref{prop:loreinst}, Remark \ref{rmk:isom} and Proposition \ref{prop:critlambda1} we can infer the following.

\begin{prop}
 Let $(M, g)$ be an $\eta$-Sasaki-Einstein manifold with constant $A$. Then a L-minimal Legendrian submanifold $L$ is Legendrian stable in $g$ if, and only if, it is Legendrian stable in the associated Lorentzian-Sasakian metric $g_\alpha$, for all $\alpha>0$. 
 \end{prop}

\subsection*{Acknowledgements} 
The authors would like to thank the referee for his/her useful comments and remarks. They were supported by the Research Training Group 1463 ``Analysis, Geometry and String Theory'' of the DFG and the first author is supported as well by the GNSAGA of INdAM. 
They also would like to thank Fabio Podest\`a for his interest in their work.

\bibliography{biblio}
\bibliographystyle{amsplain}
\end{document}